\documentclass[reqno]{amsart}%[a4paper]
\usepackage{latexsym,amsfonts,amssymb,epsfig}
\usepackage{hyperref} %% autoreferences, works only with PdfLaTeX
\usepackage{color}

\DeclareMathOperator{\ch}{char}
\DeclareMathOperator{\gr}{gr}

\DeclareMathOperator{\ad}{ad}
\DeclareMathOperator{\Der}{Der}

\DeclareMathOperator{\Lie}{Lie}
\DeclareMathOperator{\Alg}{Alg}
\DeclareMathOperator{\Jord}{Jord}
\DeclareMathOperator{\Poisson}{Poisson}

\DeclareMathOperator{\End}{End}

\newcommand{\M}{\mathrm M}      % matrix ring
\renewcommand {\limsup}{\operatorname* {\overline{lim}}}
\renewcommand {\liminf}{\operatorname* {\underline{lim}}}

\DeclareMathOperator{\GKdim}{GKdim}
\DeclareMathOperator{\LGKdim}{\underline{GKdim}}
\newcommand\dd{\partial}

\renewcommand{\a}{\alpha}

\newcommand{\Z}{\mathbb Z}            % integers
\newcommand{\R}{\mathbb R}            % reals
\newcommand{\N}{\mathbb N}            % natural numbers
\newcommand{\F}{\mathbb F}            % finite field
\newcommand\NO{\mathbb N_0}           % natural + 0
\renewcommand{\H}{\mathcal H}         % Hilbert-Poincare series
\newcommand{\QQ}{\mathbf Q}        % Lie superalgebra, first example
\newcommand{\RR}{\mathbf R}        % Lie superalgebra, second example
\newcommand{\PP}{\mathbf P}        % Poisson superalgebra
\newcommand{\JJ}{\mathbf J}        % Jordan superalgebra
\newcommand{\KK}{\mathbf K}        % second Jordan superalgebra

\renewcommand{\AA}{\mathbf A}      % associative hull of Fibonacci Lie algebra
\newcommand{\Kan}{\mathcal Kan}      % Kantor double
\newcommand{\Jor}{\mathcal Jor}      % second Kantor double
\newtheorem{Theorem}{Theorem}[section]
\newtheorem{Corollary}[Theorem]{Corollary}
\newtheorem{Lemma}[Theorem]{Lemma}
\theoremstyle{remark}
\newtheorem{Remark}{Remark}
\theoremstyle{Example}

\theoremstyle{Conjecture}
\newtheorem{Conjecture}{Conjecture}

\renewcommand{\theenumi}{\roman{enumi}}   % labels i for enumerate
\begin{document}
\title{On Jordan doubles of slow growth of Lie superalgebras}
\author{Victor Petrogradsky}
\address{Department of Mathematics, University of Brasilia, 70910-900 Brasilia DF, Brazil}
\email{petrogradsky@rambler.ru}
\thanks{The first author was partially supported by grants FAPESP~2016/18068-9, CNPq~309542/2016-2}
\author{I.P.~Shestakov}
\address{Instituto de Mathem\'atica e Estat\'istica,
Universidade de Sa\~o Paulo, Caixa postal 66281, 05315-970, Sa\~o Paulo, Brazil}
\email{shestak@ime.usp.br}
\thanks{The second author was partially supported by grants FAPESP 2014/09310-5, CNPq 303916/2014-1} %%2014/09310-5
% \\ \vspace{-0.3cm} \center{ \bf\today}
%}

\begin{abstract}
To an arbitrary Lie superalgebra $L$ we associate its Jordan double ${\mathcal Jor}(L)$,
which is a Jordan superalgebra. This notion was introduced by the second author before~\cite{She99}.
Now we study further applications of this construction.
\par
First, we show that the Gelfand-Kirillov dimension of a Jordan superalgebra can be an arbitrary number $\{0\}\cup [1,+\infty]$.
Thus, unlike the associative and Jordan algebras~\cite{KraLen,MaZe96},
one hasn't an analogue of Bergman's gap $(1,2)$ for the Gelfand-Kirillov dimension of Jordan superalgebras.
\par
Second, using the Lie superalgebra $\mathbf R$ of~\cite{PeOtto},
we construct a Jordan superalgebra $\mathbf J={\mathcal Jor}({\mathbf R})$ that is nil finely $\mathbb Z^3$-graded
(moreover, the components are at most one-dimensional), the field being of characteristic not 2.
This example is in contrast with non-existence of such examples
(roughly speaking, analogues of the Grigorchuk and Gupta-Sidki groups)
of Lie algebras in characteristic zero~\cite{MaZe99} and
Jordan algebras in characteristic not~2~\cite{Zelmanov}.
Also, $\mathbf J$ is just infinite but not hereditary just infinite.
A similar Jordan superalgebra of slow polynomial growth was constructed before~\cite{PeSh18FracPJ}.
The virtue of the present example is that it is of linear growth,
of finite width~4, namely, its $\mathbb N$-gradation by degree in the generators has components of dimensions $\{0,2,3,4\}$,
and the sequence of these dimensions is non-periodic.
\par
Third, we review constructions of Poisson and Jordan superalgebras of~\cite{PeSh18FracPJ}
starting with another example of a Lie superalgebra introduced in~\cite{Pe16}.
We discuss the notion of self-similarity for Lie, associative, Poisson, and Jordan superalgebras.
We also suggest the notion of a wreath product in case of Jordan superalgebras.
\end{abstract}
\maketitle

%%%%%%%%%%%%%%%%%%%%%%%%%%%%%%%%%%%%%%%%%%%%%%%%%%%%%%%%%%%%%%%%%%%%%%%%%%%%%%%%%%%%%%%%%%%
\section{Introduction: Superalgebras, constructions}

Denote $\NO=\{0,1,2,\dots\}$.
By $K$ denote the ground field of characteristic $\ch K\ne 2$, $\langle S\rangle_K$ a linear span of a subset $S$ in a $K$-vector space.

\subsection{Associative and Lie superalgebras}
Superalgebras appear naturally in physics and mathematics~\cite{Kac77,Scheunert,BMPZ}.
Let $\Z_2=\{\bar 0,\bar 1\}$, the group of order 2.
A {\em superalgebra} $A$ is a $\Z_2$-graded algebra $A=A_{\bar 0}\oplus A_{\bar 1}$.
The elements $a\in A_\alpha$ are called {\em homogeneous of degree} $|a|=\alpha\in\Z_2$.
The elements of $A_{\bar 0}$ are {\em even}, those of $A_{\bar 1}$ {\em odd}.
In what follows, if $|a|$ enters an expression,
then it is assumed that $a$ is homogeneous of degree $|a|\in\Z_2$,
and the expression extends to the other elements by linearity.
Let $A,B$ be superalgebras, a {\em tensor product} $A\otimes B$ is a superalgebra
whose space is the tensor product of the spaces $A$ and $B$ with the induced $\Z_2$-grading and the product satisfying Kaplansky's rule:
$$
(a_1\otimes b_1) (a_2\otimes b_2)=(-1)^{|b_1|\cdot |a_2|}a_1a_2\otimes b_1b_2,\quad a_i\in A,\ b_i\in B.
$$

An {\em associative superalgebra} $A$ is just a
$\Z_2$-graded associative algebra $A=A_{\bar 0}\oplus A_{\bar 1}$.
Let $V=V_{\bar 0}\oplus V_{\bar 1}$ be a $\Z_2$-graded vector space.
Then $\End(V)$ is an associative superalgebra, where $\End(V)_{a}=\{\phi\in\End (V)| \phi(V_{b})\subset V_{a+b},b\in\Z_2\}$, $a\in \Z_2$.
In case $\dim V_{\bar 0}=m$, $\dim V_{\bar 1}=k$ this superalgebra is denoted by $\M(m|k)$.
One has an isomorphism of superalgebras $\M(a|b)\otimes \M(c|d)\cong \M(ac+bd|ad+bc)$ for all $a,b,c,d\ge 0$.

A {\em Lie superalgebra} is a $\Z_2$-graded algebra $L=L_{\bar 0}\oplus L_{\bar 1}$ with an
operation $[\ ,\ ]$ satisfying the axioms ($\ch K\ne 2,3$):
\begin{itemize}
\item
$[x,y]=-(-1)^{|x|\cdot |y| }[y,x]$,\qquad\qquad (super-anticommutativity);
\item
$[x,[y,z]]=[[x,y],z]+(-1)^{|x|\cdot| y|}[y,[x,z]]$,\qquad (super Jacobi identity).
\end{itemize}

\subsection{Poisson superalgebras}\label{SSPoisson}
A $\Z_2$-graded vector space $A=A_{\bar 0}\oplus A_{\bar 1}$ is called a {\it Poisson superalgebra}
provided that, beside the addition, $A$ has two $K$-bilinear operations as follows:
\begin{itemize}
\item
$A=A_{\bar 0}\oplus A_{\bar 1}$ is an associative superalgebra with unit whose multiplication is denoted by
$a\cdot b$ (or $ab$), where $a, b\in A$.
We assume that $A$ is {\it supercommutative}, i.e. $a\cdot b=(-1)^{|a| |b|}b\cdot a$,
for all $a,b\in A$.
\item
$A=A_{\bar 0}\oplus A_{\bar 1}$
is a Lie superalgebra whose product is traditionally denoted by the {\it Poisson bracket}
$\{a, b\}$, where $a, b\in A$.
\item these two operations are related by the {\it super Leibnitz rule}:
\begin{equation*}
\{a\cdot b, c\}=a\cdot\{b, c\}+(-1)^{|b|\cdot |c|}\{a, c\}\cdot b,\qquad  a, b, c \in A.
\end{equation*}
\end{itemize}
Let $L$ be a Lie superalgebra,
$\{U_n| n\ge 0\}$ the natural filtration of its universal enveloping algebra $U(L)$ by degree in $L$.
Consider the {\it symmetric algebra}
$S(L)=\gr U(L)=\mathop{\oplus}\limits_{n=0}^\infty U_{n}/U_{n+1}$ (see~\cite{Dixmier}).
Recall that $S(L)$ is identified with the supercommutative superalgebra $K[v_i\,|\, i\in I]\otimes \Lambda (w_j,|\,j\in J)$, where
$\{v_i\,|\, i\in I\}$, $\{w_j\,|\, j\in J\}$, are bases of $L_{\bar 0}$, $L_{\bar 1}$, respectively.
Define a Poisson bracket by setting $\{v,w\}=[v,w]$, where $v,w\in L$,
and extending to the whole of $S(L)$ by linearity and the Leibnitz rule.
Then, $S(L)$ is turned into a Poisson superalgebra, called the {\it symmetric algebra} of $L$.
Let $L(X)$ be the free Lie superalgebra generated by a graded set $X$,
then $S(L(X))$ is a free Poisson superalgebra~\cite{Shestakov93}.

Let us consider one more example.
Let $H_{n}=\Lambda(x_1,\dots,x_n,y_1,\dots,y_n)$
be the Grassmann superalgebra supplied with a bracket determined by:
$\{x_i,y_j\}=\delta_{i,j}$, $\{x_i,x_j\}=\{y_i,y_j\}=0$ for $1\le i,j\le n$.
%(Fix $i\in\{1,\dots,n\}$, define $D_{x_i}(x_j)=\delta_{i,j}$ and $D_{x_i}(y_j)=0$, using Leibnitz rule
%this action uniquely extends to a superderivation $D_{x_i}\in\Der H_n=\WW(2n)$.)
We obtain the simple {\it Hamiltonian Poisson superalgebra} with the bracket:
\begin{equation}\label{poissonHn}
\{f,g\}=(-1)^{|f|-1}\sum_{i=1}^n
\bigg(\frac{\partial f}{\partial x_i}\frac{\partial g}{\partial y_i}
+\frac{\partial f}{\partial y_i}\frac{\partial g}{\partial x_i}\bigg),\qquad f,g\in H_{n}.
\end{equation}
Let $A$, $P$ be Poisson superalgebras, their tensor product $A\otimes P$ is a Poisson superalgebra with operations:
\begin{itemize}
\item
$(a\otimes v)\cdot (b\otimes w)= (-1)^{|v||b|}ab\otimes vw$,
\item
$\{a\otimes v, b\otimes w\}= (-1)^{|v||b|} (\{a, b\}\otimes vw+ ab\otimes \{v,w\})$, for all $a,b\in A$, $v,w\in P$.
\end{itemize}

\subsection{Jordan superalgebras}
While studying Jordan (super)algebras we always assume that $\ch K\ne 2$.
A {\it Jordan algebra} is an algebra $J$ satisfying the identities
\begin{itemize}
\item  $ab=ba$;
\item  $a^2(ca)=(a^2c)a$.
\end{itemize}
A {\em Jordan superalgebra} is a $\Z_2$-graded algebra  $J=J_{\bar 0} \oplus J_{\bar 1}$ satisfying the graded identities:
%$ab$, where $a,b\in J$,
\begin{itemize}
\item  $ab=(-1)^{|a||b|}ba$;
\item  $(ab)(cd)+(-1)^{|b||c|}(ac)(bd)+(-1)^{(\!|b|+|c|)|d|}(ad)(bc) \\
%%&=((ab)c)d+(-1)^{|c||d|+|b||c|}((ad)c)b+(-1)^{|a||b|+|a||c|+|a||d|+|c||d|}((bd)c)a.\\
=((ab)c)d+(-1)^{|b|(\!|c|+|d|)+|c||d|}((ad)c)b+(-1)^{|a|(\!|b|+|c|+|d|)+|c||d|}((bd)c)a.$
\end{itemize}

Let $A=A_{\bar 0}\oplus A_{\bar 1}$ be an associative superalgebra.
The same space supplied with
the product $a\circ b=\frac 12(ab+(-1)^{|a||b|}ba)$ is a Jordan superalgebra $A^{(+)}$.
A Jordan superalgebra $J$ is called {\it special} if it can be embedded into a Jordan
superalgebra of the type $A^{(+)}$.
Also, $J$ is called {\it i-special} (or {\it weakly special})
if it is a homomorphic image of a special one.

I.L.~Kantor suggested the following doubling process, which is applied to a Poisson (super)algebra $A$
and the result is  a Jordan superalgebra $\Kan(A)$~\cite{Kantor92}.
The $K$-space $\Kan(A)$ is the direct sum $A\oplus \bar A$,
where $\bar A$ is a copy of $A$,
let $a\in A$ then $\bar a$ denotes the respective element in $\bar A$.
In the construction of $\Kan(A)$,
the subspace $\bar A$ is supplied with the opposite $\Z_ 2$-grading, i.e., $|\bar a| = 1 - |a|$ for a $\Z_2$-homogeneous $a\in A$.
The multiplication $\bullet$ on $\Kan(A)$ is defined by:
\begin{align*}
a \bullet b      &= ab,\\
\bar a \bullet b &= (-1)^{|b|} \overline{ab},\\
a \bullet \bar b &= \overline{ab}, \\
\bar a \bullet \bar b &= (-1)^{|b|} \{a,b\},\qquad a,b\in A.
\end{align*}
This construction is important because it yielded a new series of finite dimensional simple Jordan superalgebras
$\Kan(\Lambda(n))$, $n\ge 2$,
where $\Lambda(n)$ is the Grassmann algebra in $n$ variables~\cite{Kantor92,KingMcCrimon92}.
\medskip

Simple finite dimensional nontrivial Jordan superalgebras over an algebraically closed field of characteristic zero
were classified~\cite{Kac77CA,Kantor92}.
Infinite-dimensional  simple $\Z$-graded Jordan superalgebras
with a unit element over an algebraically closed field of characteristic zero which components are uniformly bounded
are classified in~\cite{KacMarZel01}.
Recently, just infinite Jordan superalgebras were studied in~\cite{ZhePan17}.

\subsection{Growth}
We recall the notion of {\em growth}. Let $A$  be an associative (or Lie) algebra  generated by a finite set $X$.
Denote  by $A^{(X,n)}$ the subspace of $A$ spanned by all  monomials  in $X$
of length not  exceeding  $n$, $n\ge 0$.
In case of a Lie superalgebra of $\ch K=2$ we also consider formal squares of odd monomials of length at most $n/2$~\cite{PeOtto,PeSh18FracPJ}.
One defines an {\em  (ordinary) growth function}:
$$
\gamma_A(n)=\gamma_A(X,n)=\dim_KA^{(X,n)},\quad n\ge 0.
$$
Let $f,g:\N\to\R^+$ be eventually increasing and positive valued functions.
Write $f(n)\preccurlyeq g(n)$ if and only if there exist positive constants $N,C$ such that $f(n)\le g(Cn)$
for all $n\ge N$.
Introduce equivalence $f(n)\sim g(n)$ if and only if  $f(n)\preccurlyeq g(n)$ and $g(n)\preccurlyeq f(n)$.
By $f(n)\approx g(n)$, $n\to\infty$, denote that $\lim_{n\to\infty} f(n)/g(n)=1$.
Different generating sets of an algebra yield equivalent growth functions~\cite{KraLen}.

It is well known that the
exponential growth is the highest possible growth for finitely generated Lie and
associative algebras. A growth function $\gamma_A(n)$ is
compared with polynomial functions $n^\alpha$, $\alpha\in\R^+$, by
computing the {\em upper and lower Gelfand-Kirillov
dimensions}~\cite{KraLen}:
\begin{align*}
\GKdim A&=\limsup_{n\to\infty} \frac{\ln\gamma_A(n)}{\ln n}=\inf\{\a>0\mid \gamma_A(n)\preccurlyeq n^\a\} ;\\
\LGKdim A&=\liminf_{n\to\infty}\,  \frac{\ln\gamma_A(n)}{\ln n}=\sup\{\a>0\mid \gamma_A(n)\succcurlyeq n^\a\}.
\end{align*}
By Bergman's theorem, the Gelfand-Kirillov dimension of an associative algebra cannot belong to the interval $(1,2)$~\cite{KraLen}.
Similarly, Martinez and Zelmanov proved that there are no finitely generated Jordan algebras with Gelfand-Kirillov
dimension strictly between 1 and 2~\cite{MaZe96}.
But such a gap does not exist for Lie algebras, the Gelfand-Kirillov
dimension of a finitely generated Lie algebra can be an arbitrary number $\{0\}\cup [1,+\infty)$~\cite{Pe97}.
It is known that the construction of Golod yields associative nil-algebras of exponential growth.
Using specially chosen relations, Lenagan and Smoktunowicz
constructed associative nil-algebras of polynomial growth~\cite{LenSmo07}.

Suppose that $L$ is a Lie (super)algebra and $X\subset L$.
By $\Lie(X)$ denote the subalgebra of $L$ generated by $X$.
In case of associative, Poisson, and Jordan (super)algebras we use notations
$\Alg(X)$, $\Poisson(X)$, and $\Jord(X)$, respectively.
A grading of an algebra is called {\em fine} if
it cannot be splitted by taking a bigger grading group (see e.g.~\cite{BaSeZa01}).

Pro-$p$-groups and $\N$-graded Lie algebras cannot be simple.
Instead, one has another important notion.
A group (algebra) is {\it just infinite} if and only if it has no
non-trivial normal subgroups (ideals) of infinite index (codimension).
A group (algebra) is said {\it hereditary just infinite}
if and only if any normal subgroup (ideal) of finite index (codimension) is just infinite.
The Gupta-Sidki groups were the first in the class of periodic groups to be shown to be just infinite~\cite{GuptaSidki83A}.
Also, the Grigorchuk group is just infinite but not hereditary just infinite~\cite{Grigorchuk00horizons}.

%%%%%%%%%%%%%%%%%%%%%%%%%%%%%%%%%%%%%%%%%%%%%%%%%%%%%%%%%%%%%%%%%%%%%%%%%%%%%%%%
\section{Jordan double of Lie superalgebra}

First, we recall the construction of a double of a Lie superalgebra suggested by the second author~\cite{She99}.
The goal of the present paper is to study its different applications.

Let $L$ be an arbitrary Lie superalgebra.
Its symmetric algebra $S(L)$ has the structure of a Poisson superalgebra.
Observe, that the subspace $H\subset S(L)$ spanned by all tensors of length at least two is its ideal.
Thus, one obtains a (rather trivial) Poisson superalgebra $P(L)=S(L)/H$,
which equivalently can be obtained as a vector space endowed with two Poisson products which are nontrivial in the following cases:
\begin{equation}\label{defP}
P(L)=\langle 1\rangle \oplus L,\qquad 1\cdot x=x,\quad \{x,y\}=[x,y], \quad x,y\in L.
\end{equation}
Using Kantor's double, define a Jordan superalgebra $\Jor(L)=\Kan(P(L))$.
Equivalently, one can just take a vector space supplied with a product $\bullet$
which is nontrivial in the following cases
(see the example at the end~\cite{She99}):
\begin{equation}\label{defJor}
\Jor(L)=\langle 1\rangle \oplus L\oplus \langle \bar 1\rangle \oplus \bar L,\qquad
\bar x\bullet \bar y=[x,y], \quad x\bullet \bar 1=(-1)^{|x|}\bar 1\bullet x=\bar x,\quad x,y\in L;\quad 1\text{ the unit}.
\end{equation}

If an associative superalgebra $A$ is just infinite then
the related Jordan superalgebra $A^{(+)}$ is just infinite as well~\cite{ZhePan17}.
We establish a similar fact, for convenience of the reader we repeat our arguments, see~\cite{PeSh18FracPJ}.
\begin{Lemma}\label{Ljust-inf-Jor}
Let $L$ be a Lie superalgebra, consider the Jordan superalgebra $\Jor(L)$.
\begin{enumerate}
\item $\Jor(L)$  is just infinite if and only if $L$ is just infinite.
\item The ideal without unit $\Jor^o(L)=L\oplus \langle \bar 1\rangle \oplus \bar L$ is solvable of length 3.
\item $(a^2)^2=0$ for any $a\in\Jor^o(L)$.
%% moreover, $a^2=0$ if $a$ is additionally odd or even.
\end{enumerate}
\end{Lemma}
\begin{proof}
Let $L$ be not just infinite. Then there exists an ideal of infinite codimension $0\ne I\triangleleft L$
and $I\oplus \bar I$ is a nontrivial ideal of infinite codimension in $\Jor(L)$. Therefore, $\Jor(L)$ is not just infinite.

Conversely, suppose that $L$ is just infinite.
By way of contradiction,
assume that $0\ne H\subset \Jor(L)$ is an ideal of infinite codimension.
Then $\tilde H=H\cap (L\oplus \bar L)\subset \Jor(L)$ is also an ideal of infinite codimension.
Denote by $H_0$ and $\bar H_1$
the projections of $\tilde H$ onto  $L$, $\bar L$, respectively ($\bar H_1$ being the copy of a subspace $H_1\subset L$).
Since $\tilde H$ is an ideal,
$\bar 1\bullet \tilde H=\bar H_0\subset \bar H_1$ and
$\bar L\bullet \tilde H=[L,H_1]\subset H_0$ and
we get $[L,H_1]\subset H_0\subset H_1\subset L$.
Hence $H_0\subset L$ is an ideal, which must be either zero or of finite codimension by our assumption.
Let $H_0\subset L$ be of finite codimension then $\tilde H\subset \Jor(L)$ is of finite codimension, a contradiction.
Now assume that $H_0=0$. Then $[L,H_1]=0$ and $H_1$ is central. By taking $0\ne z\in H_1$,
we get an ideal $\langle z\rangle \subset L$ of infinite codimension, a contradiction.
Thus, $\Jor(L)$ is just infinite.

To prove the second claim we repeat the arguments of~\cite{She99}.
Denote $J=\Jor^o(L)$.
Then $J^2\subset L\oplus \bar L$, $(J^2)^2\subset L$, and $((J^2)^2)^2=0$.
Thus, $J$ is solvable of length 3.

To prove the last claim let $a\in J$, then $a=\alpha \bar 1+u_0+u_1+\bar w_0+\bar w_1$,
where $\alpha \in K$, $u_0,w_0\in L_{\bar 0}$, $u_1,w_1\in L_{\bar 1}$.
Then $a^2=\{w_1,w_1\}+ 2\alpha\bar u_0 \in L_{\bar 0}+\bar L_{\bar 0}$.
By the same computations, $(a^2)^2=0$.
\end{proof}

%%%%%%%%%%%%%%%%%%%%%%%%%%%%%%%%%%%%%%%%%%%%%%%%%%%%%%%%%%%%%%%%%%%%%%%%%%%%%%%%%%%%%%%%%%%%%
\section{On Gelfand-Kirillov dimension of Jordan superalgebras}

In~\cite{PeSh18FracPJ} we constructed
a Jordan superalgebra $\KK$ whose Gelfand-Kirillov dimension belongs to $(1,2)$,
that is not possible for associative and Jordan algebras~\cite{KraLen,MaZe96}.
The goal of this section is to prove a more general fact that the gap $(1,2)$
does not exist for Jordan {\it super}algebras (Theorem~\ref{TgapSJordan}).

\begin{Lemma}\label{LJor}
Let $K$ be a field, $\ch K\ne 2$.
Suppose that a Lie superalgebra $L$ is
$\Z^k$-graded by multidegree in a generating set $X=\{x_1,\ldots,x_k\}$. Let
$L=\mathop{\oplus}\limits _{n= 1}^\infty L_n$ be the
$\N$-gradation by total degree in the generators.
Consider the Jordan double $J=\Jor(L)$. Then
\begin{enumerate}
\item
one obtains $J=\mathop{\oplus}\limits_{n=0}^\infty J_n$, the $\NO$-gradation by degree in the generating set $\bar X=X\cup \{\bar 1\}$, where
\item
$J_0=\langle 1\rangle$, $J_1=L_1\oplus \langle \bar 1\rangle$, and the remaining components are as follows:
$$J_{3n-2}=L_n,\quad J_{3n-1}=\bar L_n,\quad J_{3n}=0, \qquad n\ge 1. $$
\item $J$ is $\Z^{k+1}$-graded by multidegree in $\bar X$.
\end{enumerate}
\end{Lemma}
\begin{proof}
The Jordan superalgebra $J=\Jor(L)$ is generated by $\bar X=X\cup \{\bar 1\}$.
Clearly, the $\N$-grading of $L$ by degree in $X$ extends to $J$ as well.
Let $J_{n,k}$ denote the space of Jordan monomials that include $n$ letters from $X$ and $k$ letters $\bar 1$, where $n,k\ge 0$.
Let us prove that
\begin{equation}\label{jnk}
J_{n,k}=
\begin{cases}
  \langle 1\rangle, & n=k=0;\\
  \langle \bar 1\rangle, & n=0,\ k=1;\\
\delta_{k,2n-2}L_{n}, &      k \text { even};\\
\delta_{k,2n-1}\bar L_{n}, & k  \text { odd};
%%  0, & 2n-k\notin\{1,2\};
\end{cases}\qquad\qquad n,k\ge 0.
\end{equation}
We proceed by induction on $l=n+k$.
The base of induction is $l=0$ and  $l=1$, in which cases we have $J_{0,0}=\langle 1\rangle $,
$J_{1,0}=L_{1}=\langle X\rangle$ and  $J_{0,1}=\langle \bar 1\rangle$.
Assume that $l\ge 2$.
Observe that to have a nonzero component we need at least one letter $\bar 1$, thus $k\ge 1$.
Let $k\ge 1$ be odd.
Since $\bar{\quad}$ yields a $\Z_2$-grading we have
$J_{n,k}\subset \bar L$, such elements can appear only as the products $\bar 1\bullet  L$.
Using the inductive assumption,
$$J_{n,k}=\bar 1\bullet J_{n,k-1}=\bar 1\bullet \delta_{k-1,2n-2} L_{n}=\delta_{k-1,2n-2}\bar L_{n}=\delta_{k,2n-1}\bar L_n.$$
Let $k\ge 2$ be even.
As above, $J_{n,k}\subset L$ and such elements can appear only as the products $\bar L \bullet \bar L$.
Using the inductive assumption, we have
$$J_{n,k}
=\sum_{\substack{n_1+n_2=n\\ k_1+k_2=k\\ k_1,k_2 \text { odd}}} J_{n_1,k_1} \bullet J_{n_2,k_2}
=\sum_{\substack{n_1+n_2=n\\ k_1+k_2=k\\ k_1=2n_1-1\\ k_2=2n_2-1}} \bar L_{n_1} \bullet \bar L_{n_2}
=\sum_{\substack{n_1+n_2=n\\k=2n-2} } [L_{n_1},L_{n_2}]=\delta_{k,2n-2}L_n.
$$
The inductive step is proved.

By~\eqref{jnk},
we obtain a direct sum $J=\mathop{\oplus}\limits_{n,k\ge 0}J_{n,k}$,
which is a $\Z^2$-gradation of $J$ by the respective multidegree in $X\cup \{\bar 1\}$.
By setting $J_m=\mathop{\oplus}\limits_{n+k=m} J_{n,k}$, $m\ge 0$, we get the claimed $\NO$-gradation.
Counting the total degree in~\eqref{jnk}, we get~(ii).
The multidegree $\Z^{k+1}$-gradation is proved similarly.
\end{proof}
Define generating functions:
\begin{align*}
\H(L,t)&=\sum_{n=1}^\infty \dim L_n t^n,\\
\H(J,t_1,t_2)&=\sum_{n,m=0}^\infty \dim J_{n,m} t_1^n t_2^m,\\
\H(J,t)&=\sum_{n=0}^\infty \dim J_n t^n=\H(J,t,t).
\end{align*}
\begin{Corollary} \label{Ccomp}
Using notations above
\begin{enumerate}
\item
\begin{align*}
\H(J,t_1,t_2)&=1+t_2+\bigg(\frac 1{t_2}+\frac 1{t_2^2}\bigg)\H(L,t_1t_2^2);\\
\H(J,t)&=1+t+\bigg(\frac 1{t}+\frac 1{t^2}\bigg)\H(L,t^3).
\end{align*}
\item We obtain an equivalent growth function: $\gamma_J(\bar X,n)\sim \gamma_L(X,n)$.
\item $J$ has the same (lower and upper) Gelfand-Kirillov dimensions as those for $L$ (provided that they exist).
\item Let $\gamma_L(X,n)\approx Cn^r$, $n\to \infty$ for some constants $C>0$, $r\ge 1$,
moreover assume that $C_1n^{r-1}\le \dim L_n\le C_2n^{r-1}$, $n\ge n_0$, for some constants $n_0,C_1,C_2$.
Then $\gamma_J(\bar X,n)\approx  2C(n/3)^r$, $n\to\infty$.
\end{enumerate}
\end{Corollary}
\begin{proof}
Using~\eqref{jnk} we get the formulas for the generating functions.
We have the growth functions $\gamma_L(n,X)=\sum_{k=1}^n\dim L_n$, $n\ge 1$, and
$\gamma_J(n,\bar X)=\sum_{k=0}^n\dim J_n$, $n\ge 0$.
By Lemma~\ref{LJor} (ii), we get
\begin{align*}
  \gamma_{J}(\bar X, 3m)=\gamma_{J}(\bar X, 3m-1)& =2+2\gamma_L(X,m); \\
  \gamma_{J}(\bar X,3m-2)&=2+2\gamma_L(X,m)-\dim L_m,\qquad m\ge 1.
\end{align*}
Now it remains to use the polynomial estimates on the growth of $L$.
\end{proof}
\begin{Theorem}\label{TgapSJordan}
Let $K$ be a field, $\ch K\ne 2$.
Fix any real number $r\ge 1$.
There exists a three generated Jordan superalgebra $J$ with the following properties.
\begin{enumerate}
\item $\GKdim J=\LGKdim J=r$;
\item $J$ is graded by degree in the generators, we have $J=\mathop{\oplus}\limits_{n=0}^\infty J_n$ where
$J_{3n}=0$ for all $n\ge 1$;
\item its ideal without unit $J^0$ is solvable of length 3.
\end{enumerate}
\end{Theorem}
\begin{proof}
The first author constructed a 2-generated Lie algebra $L$ such that $\GKdim L=\LGKdim L=r$ and
$L$ is $\NO^2$-graded by multidegree in the generators,
moreover, we have estimates $C_1n^{r-1}\le \dim L_n\le C_2n^{r-1}$, $n\ge 1$, for some constants $C_1,C_2$~\cite{Pe97}.
Also $(L^2)^3=0$, this condition can be written as
$L\in \mathbf N_2\mathbf A$ using notations of varieties of Lie algebras~\cite{Ba}.

Now we consider $J=\Jor(L)$ and apply Lemma~\ref{LJor}, Corollary~\ref{Ccomp}, and
Lemma~\ref{Ljust-inf-Jor}.
\end{proof}

%%%%%%%%%%%%%%%%%%%%%%%%%%%%%%%%%%%%%%%%%%%%%%%%%%%%%%%%%%%%%%%%%%%%%%%%%%%%%
\section{Just infinite nil Jordan superalgebra of finite width}

The Grigorchuk and Gupta-Sidki groups play fundamental role in modern group theory~\cite{Grigorchuk80,GuptaSidki83}.
They are natural examples of self-similar finitely generated periodic groups.
We discuss their analogues in different classes of algebras.
First, we recall their properties.

\subsection{Examples of infinite $p$-groups, their width, and related Lie rings}
Let $G$ be a group and $G=G_1\supset G_2\supset \cdots$ a {\it central series}, i.e. $(G_i,G_j)\subset G_{i+j}$ for all $i,j\ge 1$.
One constructs a related $\N$-graded Lie ring $L=\oplus_{i\ge 1} L_i$, where $L_i= G_i/G_{i+1}$, $i\ge 1$.
A product is given by $[a G_{i+1},b G_{j+1}]=(a,b)G_{i+j+1}$,
where $a\in G_i$, $b\in G_j$, and $(a,b)=a^{-1}b^{-1}ab$.
In particular, consider the {\it lower central series} $\gamma_n(G)$, $n\ge 1$ formed by the commutator subgroups.
One obtains the Lie ring
$$
L(G)=\mathop{\oplus}\limits_{n\ge 1} \gamma_n(G)/\gamma_{n+1}(G).
$$
Let $\Delta$ be the {\it augmentation ideal} of the group ring $K[G]$.
One defines the {\it dimension subgroups} $G_n=\{g\in G\mid 1-g\in \Delta^n\}$, $n\ge 1$,
they yield a central series.
Assume that $\ch K=p>0$. Then $(G_n|n\ge 1)$ is also called a {\it lower central $p$-series},
its factors are elementary abelian $p$-groups, and
one obtains a restricted Lie algebra over the prime field $\F_p$~\cite{Passi}:
$$
L_{\F_p}(G)=\mathop{\oplus}\limits_{n\ge 1} G_n/G_{n+1}.
$$
One also defines $L_{R}(G)=L_{\F_p}(G)\otimes_{\F_p} R$, where $R$ is a commutative ring.
Structural constants of both Lie rings $L(G)$ and $L_{\F_p}(G)$ of the Grigorchuk group $G$ are computed in~\cite{BaGr00}.
Also, Bartholdi presented $L_{\F_2}(G)$ as a self-similar restricted Lie algebra~\cite{Bartholdi15}.

A graded Lie ring $L=\oplus_{i\ge 1} L_i$ is said {\it nil graded} if
for any homogeneous element $x\in L_i$, $i\ge 1$, the mapping $\ad x$ is nilpotent.
In case of the Grigorchuk group $G$, the restricted Lie algebra $L_{\F_2}(G)$
is nil graded because the group $G$ is periodic~\cite{BaGr00,Zel93}.
There is also a similar general result: the Lie ring $L(G)$ associated to the lower central series of a finitely
generated residually-$p$ torsion group is nil graded~\cite{MaZe16}.
Thus, the Lie ring $L(G)$ of the Grigorchuk group $G$ is also nil graded.
On the other hand, the following ring extension of the restricted Lie algebra related to the dimension series
$L_{\F_2[x,y]}(G)$ of the Grigorchuk group $G$ is no longer nil graded, namely, the first component has non-nil elements~\cite{SmaZel06}.
(Those arguments  also one can apply to the Lie ring $L(G)$ while extending scalars to $\Z[x,y]$).
Bartholdi established a stronger fact that
the restricted Lie algebra $L_{\F_2}(G)$ over the prime field $\F_2$
is nil graded while $L_{\F_4}(G)$ (i.e. extension to the field $\F_4$) is no longer nil graded~\cite{Bartholdi15}.

A residually $p$-group $G$ is said to be of {\it finite width} if
all factors $G_i/G_{i+1}$ are finite groups with uniformly bounded orders.
The Grigorchuk group $G$ is of finite width, namely,
$\dim_{\F_2} G_i/G_{i+1}\in\{1,2\}$ for $i\ge 2$~\cite{Rozh96,BaGr00}.
Thus, the respective Lie algebra $L_{\F_2}(G)=\oplus_{i\ge 1} L_i$ is of linear growth.
On the other hand, the Gupta-Sidki group is of infinite width~\cite{Bartholdi05}.

Similarly, one defines a {\it width} of a residually nilpotent Lie (super)algebra $L$
as the maximum of dimensions of the lower central series factors $L^n/L^{n+1}$, $n\ge 1$.
Lie algebras of finite width over a field of positive characteristic
and a possibility of their classification under additional conditions are discussed in~\cite{ShaZel97,ShaZel99}.
Infinite dimensional $\N$-graded Lie algebras $L=\mathop{\oplus}\limits_{n=1}^\infty L_n$
with one-dimensional components in characteristic zero were classified by Fialowski~\cite{Fial83}.

\subsection{Nonexistence of nil graded Lie algebras in characteristic zero and Jordan algebras of characteristic not 2}
Since the Grigorchuk group is of finite width, probably,
a "right analogue" of it should be a Lie algebra of finite width having $\ad$-nil elements.
In the next result, the components are of bounded dimension and consist of $\ad$-nil elements.
Informally speaking, there are no "natural analogs" of the Grigorchuk and Gupta-Sidki groups
in the world of {\it Lie algebras of characteristic zero},
we say this strictly in terms of the following result.
\begin{Theorem}[{Martinez and Zelmanov~\cite{MaZe99}}]
\label{TMarZel}
Let $L=\oplus_{\a\in\Gamma}L_\alpha$ be a Lie algebra over a field $K$ of
characteristic zero graded by an abelian group $\Gamma$. Suppose that
\begin{enumerate}
\item
there exists $d>0$ such that $\dim_K L_\alpha \le d $ for all $\alpha\in\Gamma$,
\item
every homogeneous element $a\in L_\a$, $\a\in\Gamma$, is ad-nilpotent.
\end{enumerate}
Then the Lie algebra $L$ is locally nilpotent.
\end{Theorem}

Strictly in terms of the next result, we say again that
there are no "natural analogs" of the Grigorchuk and Gupta-Sidki groups in the class of {\it Jordan algebras} too.

\begin{Theorem}[{Zelmanov, private communication~\cite{Zelmanov}}]\label{TZelmanov}
Jordan algebras satisfy a verbatim analogue of Theorem~\ref{TMarZel} over a field $K$, $\ch K\ne 2$.
\end{Theorem}

\subsection{Examples of nil restricted Lie algebras and Lie superalgebras in arbitrary chracteristics}
The first author constructed an analogue of the Grigorchuk group
in case of restricted Lie algebras of characteristic~2~\cite{Pe06},
Shestakov and Zelmanov extended this construction to an arbitrary positive characteristic~\cite{ShZe08}.
Thus, we have examples of finitely generated restricted Lie algebras with a nil $p$-mapping.
See further constructions in~\cite{PeSh09,PeSh13fib,PeShZe10,Bartholdi15}.
A family of restricted Lie algebras of slow growth with a nil $p$-mapping
was constructed in~\cite{Pe17}, in particular, we construct a continuum
subfamily of such algebras with Gelfand-Kirillov dimension one
but their growth is not linear.
As a particular case, we get a nil $\Z_2$-graded Lie algebra of width~2 over a field of characteristic~2,
which $p$-hull has width~3 and a non-nil $p$-mapping~\cite{PeOtto}.

In contrast with Theorem~\ref{TMarZel},
there are natural analogs of the Grigorchuk and Gupta-Sidki groups
in the world of {\it Lie  superalgebras} of {\it arbitrary characteristics}~\cite{Pe16},
where two Lie superalgebras were constructed,
the most interesting cases being those of characteristic zero,
see further examples in~\cite{PeOtto,PeSh18FracPJ}.
In all these examples (actually, four examples), $\ad a$ is nilpotent,
$a$ being an even or odd element with respect to the corresponding $\Z_2$-gradings.
This property is an analogue of the periodicity of the Grigorchuk and Gupta-Sidki groups.
The second Lie superalgebra $\QQ$ in~\cite{Pe16} has a natural fine $\Z^3$-gradation
with at most one-dimensional components.
In particular, $\QQ$ is a nil finely $\Z^3$-graded Lie superalgebra, which shows
that an extension of Theorem~\ref{TMarZel}
for the Lie {\it super}algebras of characteristic zero is not valid.
\medskip

\subsection{Examples of nil graded Lie and Jordan superalgebras of finite width}
The Jordan superalgebra $\KK$ constructed in~\cite{PeSh18FracPJ} shows
that an extension of Theorem~\ref{TZelmanov} for the {\it Jordan superalgebras}, characteristic being not 2, is not valid.
The goal of the present section is to provide
a similar but "smaller" example, namely, a nil graded Jordan superalgebra $\JJ$ of {\it finite width} (Theorem~\ref{TJOtto}).
These facts resemble those for Lie algebras and superalgebras mentioned above.
\medskip

Both Lie superalgebras of~\cite{Pe16} and the Lie superalgebra of~\cite{PeSh18FracPJ} are of infinite width.
Now, we shall use the Lie superalgebra of finite width constructed in~\cite{PeOtto}.

Let $\Lambda=\Lambda(x_i\vert i\geq 0)$ be the Grassmann algebra.
The Grassmann variables and respective superderivatives
$\lbrace x_i,\dd_i\mid i\geq 0\rbrace$ are odd elements of the superalgebra $\End\Lambda$.
Define so called {\it pivot elements}:
\begin{equation}\label{pivot}
v_i=\dd_i+x_{i}x_{i+1}(\dd_{i+2}+x_{i+2}x_{i+3}(\dd_{i+4}+x_{i+4}x_{i+5}(\dd_{i+6}+\ldots)))\in\Der\Lambda,\qquad i\geq 0.
\end{equation}
Define the {\em shift} mappings:
\begin{equation*}
\tau(x_i)=x_{i+1},\quad \tau(\dd_i)=\dd_{i+1},\quad \tau(v_i)=v_{i+1},\qquad i\ge 0.
\end{equation*}

Consider the Lie superalgebra $\RR=\Lie(v_0,v_1)\subset\Der\Lambda$ and its associative hull
$\AA=\Alg(v_0,v_1)\subset \End \Lambda$.
We formulate their main properties, for more details see the original paper~\cite{PeOtto}.

\begin{Theorem}[{O. de Morais Costa, V. Petrogradsky~\cite{PeOtto}}]\label{TOtto}
Consider the Lie superalgebra $\RR=\Lie(v_0,v_1)$ and its associative hull $\AA=\Alg(v_0,v_1)$, where $\ch K\ne 2$.   Then
\begin{enumerate}
\renewcommand{\theenumi}{\roman{enumi}}
\item $\RR$ has a monomial basis consisting of standard monomials of two types.
\item $\RR$ and $\AA$ are $\Z^2$-graded by multidegree in the generators.
      Also, $\RR$ has the degree $\N$-gradation,
      which components are isomorphic to the factors of the lower central series.
\item We put the basis monomials of $\RR$ and $\AA$ at lattice points of plane $\Z^2\subset\R^2$ using the multidegree.
      These monomials are in regions of plane bounded by pairs of logarithmic curves.
\item The components of the $\Z^2$-grading of $\RR$ are at most one-dimensional,
      thus, the $\Z^2$-grading of $\RR$ is fine.
\item $\GKdim\RR=\LGKdim\RR=1$, moreover, $\RR$ has a linear growth and the growth function satisfies
      $\gamma_\RR(m)\approx 3m$ as $m\to\infty$.
\item Moreover, $\RR$ is of finite width 4. Namely,
      let $\RR=\mathop{\oplus}\limits_{n=1}^\infty\RR_n$ be the $\N$-grading by degree in the generators,
      where $\RR_n\cong \RR^n/\RR^{n+1}$, $n\ge 1$, are the lower central series factors (see (ii)).
      Then  the coefficients $(\dim\RR_n| n\ge 1)$, are $\{2,3,4\}$.
      This sequence is non-periodic.
\item $\GKdim\AA=\LGKdim\AA=2$.
\item %(analogue of periodicity)
      Homogeneous elements of the grading $\mathbf{R}=\mathbf{R}_{\bar 0}\oplus\mathbf{R}_{\bar 1}$ are $\ad$-nilpotent.
\item $\RR$ is just infinite  but not hereditary just infinite.
\item $\RR$ again shows that an extension of~Theorem~\ref{TMarZel} (Martinez and Zelmanov~\cite{MaZe99})
      to the Lie superalgebras of characteristic zero is not valid.
\end{enumerate}
\end{Theorem}
\begin{Remark}
The first counterexample of a nil finely $\Z^3$-graded Lie superalgebra
of slow polynomial growth in any characteristic was suggested before (the second Lie superalgebra $\QQ$ in~\cite{Pe16}).
The virtue of the nil finely $\Z^2$-graded Lie superalgebra $\RR$ above
is that it is of linear growth, moreover, of finite width 4, and just infinite.
Claim~(vi)  is analogous to the fact that the Grigorchuk group is of finite width~\cite{Rozh96,BaGr00}.
Thus, $\RR$ is a "more appropriate" analogue of the Grigorchuk group than the second Lie superalgebra $\QQ$ of~\cite{Pe16}
or the Lie superalgebra considered recently in~\cite{PeSh18FracPJ}, both being of infinite width
(because their Gelfand-Kirillov dimensions are bigger than 1).
\end{Remark}

Actually, the non-periodicity of the sequence above (claim~(vi)) was proved only in case $p=2$~\cite{PeOtto},
but the proof of non-periodicity for other characteristics is similar.
\medskip

Now, we construct the following Jordan superalgebra and describe its properties.
We use the Lie superalgebra $\RR$ described above and take its Jordan double $\JJ=\Jor(\RR)$.

\begin{Theorem}\label{TJOtto}
Let $\ch K\ne 2$ and $\RR=\Lie(v_0,v_1)$ be the Lie superalgebra of Theorem~\ref{TOtto}.
Consider the Jordan superalgebra $\JJ=\Jor(\RR)$ and its subalgebra without unit $\JJ^o$.
They have the following properties.
\begin{enumerate}
\renewcommand{\theenumi}{\roman{enumi}}
\item $\JJ$ is $\Z^3$-graded by multidegree in $\bar X=\{v_0,v_1,\bar 1\}$.
\item We put monomials of $\JJ$ at lattice points of plane $\Z^2\subset\R^2$ using the partial multidegree in $\{v_0,v_1\}$.
      These monomials are  bounded by a pair of logarithmic curves.
\item The components of the $\Z^3$-grading of $\JJ$ are at most one-dimensional.
\item $\GKdim \JJ=\LGKdim\JJ=1$, moreover, $\JJ$ is of linear growth and
      $\gamma_\JJ(\bar X,m)\approx 2m$, as $m\to\infty$.
\item Let $\JJ=\mathop{\oplus}\limits_{n=0}^\infty\JJ_n$ be the $\NO$-grading by degree in $\bar X$.
      Then $\JJ$ is of finite width 4, and
      the coefficients $(\dim\JJ_n| n\ge 1)$ are $\{0,2,3,4\}$, where the trivial components are $\JJ_{3m}=0$, $m\ge 1$.
      This sequence is non-periodic.
%%      (This is analogous to the fact that the Grigorchuk group is of finite width).
\item $\JJ$ is just infinite  but not hereditary just infinite.
\item $(a^2)^2=0$ for any $a\in\JJ^o$.
\item Let $a\in\JJ^o$ and $a\in J_{n,m,k}$, $(n,m,k)\in\Z^3$, then $a^2a=aa^2=0$. Thus, $\JJ^o$ is nil $\Z^3$-graded.
      %%%%%moreover, $a^2=0$ provided that $a$ is additionally odd or even.
\item $\JJ$ again shows that an extension of~Theorem~\ref{TZelmanov} of Zelmanov~\cite{Zelmanov}
      to the Jordan superalgebras of characteristic not 2 is not valid.
\end{enumerate}
\end{Theorem}
\begin{proof}
Almost all statements follow from the properties of $\RR$ described in Theorem~\ref{TOtto}
by applying Lemma~\ref{Ljust-inf-Jor},  Lemma~\ref{LJor},  and Corollary~\ref{Ccomp}.
The fact that $\JJ$ is not hereditary just infinite is proved as the same property of the second Jordan superalgebra $\KK$
in~\cite[Theorem~13.4]{PeSh18FracPJ}.
Let $a\in J_{n,m,k}$, recall that the latter component is one-dimensional.
If $a\in L$ then $a^2=0$. Let $a\in \bar L$,  then the square is again zero except the pivot elements, namely,
let $a=\bar v_n$ then $a^2=\bar v_n\bullet \bar v_n=\{v_n,v_n\}=x_{n+1}v_{n+2}\in L$  (see~\cite{PeOtto})
and  $a^2a=aa^2=0$ by~\eqref{defJor}.
\end{proof}
\begin{Remark}
A similar example of a just-infinite nil finely $\Z^4$-graded Jordan superalgebra
of slow polynomial growth  was suggested before, see the second Jordan superalgebra $\KK$ in~\cite{PeSh18FracPJ}.
But, the present example is a more "appropriate analogue" of the Grigorchuk group,
because the present Jordan superalgebra $\JJ$ is of linear growth, moreover, of finite width 4.
This property resembles the finite width of the Grigorchuk group.
\end{Remark}

\begin{Remark}
The example $\JJ$ above shows that just infinite $\Z$-graded Jordan superalgebras of finite width can have
a complicated structure unlike the classification of such simple algebras over an algebraically closed field of characteristic zero~\cite{KacMarZel01}.
\end{Remark}
%%%%%%%%%%%%%%%%%%%%%%%%%%%%%%%%%%%%%%%%%%%%%%%%%%%%%%%%%%%%%%%%%%%%%%%%%%%%%%%%%%%%%%%%%%%%%%%%%%%%
\section{On self-similarity of superalgebras}\label{Sself-similar}

\subsection{Self-similarity of Lie superalgebras}
We say that an algebra is {\it fractal} provided that it contains infinitely many copies of itself.
In this section we discuss the notion of self-similarity for our superalgebras.
The notion of self-similarity plays an important role in group theory~\cite{Grigorchuk00horizons,Nekr05}.
The Fibonacci Lie algebra introduced by the first author is "self-similar"~\cite{PeSh09} but not in terms of
the definition of self-similarity given by Bartholdi~\cite{Bartholdi15}.
Namely, a Lie algebra $L$ is called {\it self-similar} if it affords a homomorphism~\cite{Bartholdi15}:
\begin{equation}\label{selfQ}
\psi:L\rightarrow\Der A\rightthreetimes (A\otimes L),
\end{equation}
where $A$ is a commutative algebra,
$\Der A$ the Lie algebra of derivations, the product of the right hand side is defined via
the natural action of $\Der A$ on $A$.
The first author constructed a family of two-generated restricted Lie algebras with a nil $p$-mapping determined by two infinite sequences,
if these sequences are periodic we get self-similar restricted Lie algebras~\cite{Pe17}.
Recently, self-similar Lie algebras were studied in~\cite{FutKochSid}.

This definition easily extends to Lie superalgebras by
setting $A$ to be a supercommutative associative superalgebra and $\Der A$ the Lie superalgebra of superderivations.
We have two original examples of ad-nil self-similar Lie superalgebras of slow polynomial growth over an arbitrary field~\cite{Pe16}.

Recall the construction of the second Lie superalgebra of~\cite{Pe16}.
Let $\ch K\ne 2$ and $\Lambda=\Lambda[x_i,y_i,z_i| i\ge 0]$ the Grassmann superalgebra,
denote the respective partial superderivatives as  $\{\partial_{x_i},\partial_{y_i},\partial_{z_i}| i\ge 0\}$.
Define series of elements, called the {\it pivot elements}, of the Lie superalgebra of superderivations $\Der\Lambda$:
\begin{equation}\label{aibici0}
\begin{split}
a_i &= \partial_{x_i} + y_ix_i(\partial_{x_{i+1}}+ y_{i+1}x_{i+1}(\partial_{x_{i+2}} +y_{i+2}x_{i+2}(\partial_{x_{i+3}}+ \cdots  ))),\\
b_i &= \partial_{y_i} + z_iy_i(\partial_{y_{i+1}}+ z_{i+1}y_{i+1}(\partial_{y_{i+2}} +z_{i+2}y_{i+2}(\partial_{y_{i+3}}+ \cdots  ))),\\
c_i &= \partial_{z_i} + x_iz_i(\partial_{z_{i+1}}+ x_{i+1}z_{i+1}(\partial_{z_{i+2}} +x_{i+2}z_{i+2}(\partial_{z_{i+3}}+ \cdots  ))),
\end{split}
\qquad i\ge 0.
\end{equation}
Define the {\em shift} mapping $\tau:\Lambda\to \Lambda$ and its natural extensions to the elements defined above:
\begin{align*}
\begin{split}
\tau(x_i)&=x_{i+1}, \quad\  \tau(y_i)=y_{i+1},\quad \quad \tau(z_i)=z_{i+1},\\
\tau(\partial_{x_i})&=\partial_{x_{i+1}}, \quad \tau(\partial_{y_i})=\partial_{y_{i+1}},\quad \tau(\partial_{z_i})=\partial_{z_{i+1}},\\
\tau(a_i)&=a_{i+1}, \quad \ \tau(b_i)=b_{i+1},\quad \quad \tau(c_i)=c_{i+1},\\
\end{split}
\qquad\quad i\ge 0.
\end{align*}

Define the Lie superalgebra $\QQ=\Lie(a_0,b_0,c_0)\subset \Der\Lambda$
and its associative hull $\AA=\Alg(a_0,b_0,c_0)\subset \End \Lambda$.
For more details see the original paper~\cite{Pe16}. We observe only the following.

\begin{Lemma} The Lie superalgebra $\QQ=\Lie(a_0,b_0,c_0)$ defined above is self-similar with a natural self-similarity embedding:
$$
\psi:\QQ\hookrightarrow \langle \partial_{x_0},\partial_{y_0},\partial_{z_0}\rangle_K\rightthreetimes \Lambda[x_0,y_0,z_0]\otimes \tau(\QQ).
$$
\end{Lemma}
\begin{proof}
By~\eqref{aibici0} we obtain  a recursive presentation:
\begin{equation}\label{recursiveQ}
\begin{split}
a_0 &= \partial_{x_0} + y_0x_0 \tau(a_0),\\
b_0 &= \partial_{y_0} + z_0y_0 \tau(b_0),\\
c_0 &= \partial_{z_0} + x_0z_0 \tau(c_0).
\end{split}
\end{equation}
Observe that it is sufficient to find a desired presentation in the form~\eqref{selfQ} for the generators.
Indeed, we extend presentation~\eqref{recursiveQ} to an arbitrary $a\in\QQ=\Lie(a_0,b_0,c_0)$, using that $\tau:\QQ\to\QQ$ is a shift monomorphism.
\end{proof}

\begin{Conjecture}
Let $\RR=\Lie(v_0,v_1)$ be the Lie superalgebra, where $v_0,v_1$ are defined by~\eqref{pivot} (i.e. the example of~\cite{PeOtto}).
Then $\RR$ is not self-similar.
We conjecture that the Lie superalgebra of~\cite{PeSh18FracPJ} is not self-similar as well.
\end{Conjecture}
Indeed, recall that $\RR$ is fractal.
But, a self-similarity embedding for $\RR$ might look like:
$$\psi:\RR\hookrightarrow \langle \dd_0\rangle_K\rightthreetimes \Lambda(x_0)\otimes \tau(\RR).$$
Recall that $\RR$ is generated by $\{v_0,v_1\}$,
where the second generator is of the required form $v_1=\tau(v_0)\in \tau(\RR)$.
By~\eqref{pivot}, one has $v_0=\dd_0+x_0\cdot x_1v_2$ where $x_1v_2\notin\tau(\RR)$.
Indeed, $x_1v_2=\tau(x_0v_1)$ where $x_0v_1\notin \RR$ (see~\cite{PeOtto}).
It seems that we cannot split our variables in any similar  way.

\subsection{Self-similarity of associative superalgebras}
Similar to the case of associative algebras~\cite{Bartholdi06,Sidki09,PeSh13ass},
we say that an associative superalgebra $A$ is {\it self-similar} provided that  there exists
a {\em self-similarity embedding}:
\begin{equation}\label{selfA}
\psi:A \hookrightarrow  \M(n|m)\otimes A,
\end{equation}
for some matrix superalgebra $\M(n|m)$, the tensor product being that of associative superalgebras.
\begin{Lemma} The associative superalgebra $\AA=\Alg(a_0,b_0,c_0)$ defined above~\eqref{aibici0}
is self-similar with a self-similarity embedding:
$$
\psi:\AA\hookrightarrow   \M(4|4) \otimes \AA.
%\partial_{x_0},\partial_{y_0},\partial_{z_0}\rangle_K\rightthreetimes \Lambda[x_0,y_0,z_0].
$$
\end{Lemma}
\begin{proof}
As above, we use presentation~\eqref{recursiveQ} and get desired presentations for all elements of $\AA$.
Observe that $\Alg(\partial_{x_0},x_0)\cong \M(1|1)$, where $\M(1|1)$
is the superalgebra of $2\times 2$-matrices,
even part consists of diagonal matrices, and odd part consists of off-diagonal elements.
Next, we use that $\Alg(\partial_{x_0},x_0,\partial_{y_0},y_0,\partial_{z_0},z_0)\cong \M(1|1)\otimes \M(1|1)\otimes \M(1|1)\cong \M(4|4)$.
\end{proof}

\begin{Conjecture}
Consider the associative superalgebra $A=\Alg(v_0,v_1)$ corresponding to
the Lie superalgebra considered above $\RR=\Lie(v_0,v_1)$, where $v_0,v_1$ are defined by~\eqref{pivot} (i.e. the example of~\cite{PeOtto}).
Then $A$ is not self-similar.
We conjecture that the respective associative superalgebra of~\cite{PeSh18FracPJ} is not self-similar as well.
\end{Conjecture}

\subsection{Self-similarity of Poisson superalgebras}
Let us call a Poisson superalgebra $P$ {\it self-similar} if there exist a Poisson superalgebra $H$  and an embedding:
\begin{equation}\label{selfP}
\psi:P\hookrightarrow  H\otimes P,
\end{equation}
the tensor product being that of Poisson superalgebras.

Following constructions of~\cite{PeSh18FracPJ},
let us suggest a Poisson superalgebra related to $\QQ$ above
(this was not done in the original paper~\cite{Pe16}, we introduced the Poisson and Jordan superalgebras in case of
another example in~\cite{PeSh18FracPJ}).
Consider the Grassmann superalgebra
$H=\Lambda[x_i,y_i,z_i,X_i,Y_i,Z_i| i\ge 0]$ and supply it with the Poisson bracket, which nontrivial products are
$\{X_i,x_i\}=1$, $\{Y_i,y_i\}=1$, $\{Z_i,z_i\}=1$ for all $i\ge 0$.
Thus, $H$ is turned into a Poisson superalgebra.
We formally substitute big letters instead of respective derivatives in~\eqref{aibici0}:
\begin{equation}\label{aibiciP}
\begin{split}
A_i &= X_i + y_ix_i(X_{i+1}+ y_{i+1}x_{i+1}(X_{i+2} +y_{i+2}x_{i+2}(X_{i+3}+ \cdots  ))),\\
B_i &= Y_i\, + z_iy_i(Y_{i+1}\,+ z_{i+1}y_{i+1}(Y_{i+2}\, +z_{i+2}y_{i+2}(Y_{i+3}\,+ \cdots  ))),\\
C_i &= Z_i + x_iz_i(Z_{i+1}+ x_{i+1}z_{i+1}(Z_{i+2} +x_{i+2}z_{i+2}(Z_{i+3}+ \cdots  ))),
\end{split}
\qquad i\ge 0.
\end{equation}
We define the shift endomorphism $\tau:H\to H$ as above.
Actually, we need to consider our elements in a completion of the Poisson superalgebra $H$, see~\cite{PeSh18FracPJ}.
Now we define a Poisson superalgebra $\PP=\Poisson(A_0,B_0,C_0)$.
\begin{Lemma}
The Poisson superalgebra $\PP=\Poisson(A_0,B_0,C_0)$ is self-similar with the self-similarity embedding:
$$  \PP\hookrightarrow  H_3\otimes \tau(\PP).$$
\end{Lemma}
\begin{proof}
By~\eqref{aibiciP} we obtain  a recursive presentation:
\begin{equation}\label{recursiveP}
\begin{split}
A_0 &= X_0 + y_0x_0 \tau(A_0),\\
B_0 &= Y_0 + z_0y_0 \tau(B_0),\\
C_0 &= Z_0 + x_0z_0 \tau(C_0).
\end{split}
\end{equation}
Also, we observe that $\Poisson(x_0,y_0,z_0,X_0,Y_0,Z_0)\cong H_3$, see~\eqref{poissonHn}.
\end{proof}

\begin{Remark}
Consider a "small" Poisson superalgebra related to $\QQ$ defined by~\eqref{defP},
namely, $P(\QQ)=\langle 1\rangle \oplus \QQ$.
This algebra is fractal. But,
it seems that it is not self-similar according to our definition. Namely, there is a problem with a homomorphism
for the associative product.
\end{Remark}

\subsection{Self-similarity and wreath products of Jordan superalgebras}
We have the notion of self-similarity for Lie superalgebras~\eqref{selfQ} (a modification of that of Bartholdi~\cite{Bartholdi15}),
associative superalgebras~\eqref{selfA}, and Poisson superalgebras~\eqref{selfP}.
But we lack a similar notion of self-similarity for Jordan superalgebras.

We start with an observation.
Let $P$ be a Poisson superalgebra and $J=\Kan(P)=P\oplus \bar P$ its Kantor double, which is a Jordan superalgebra. Define a mapping
$$D:J\to J,\qquad
D(a)=0,\quad a\in P,\qquad
D(\bar a)=(-1)^{|a|} a,\quad \bar a\in \bar P.$$
One checks that $D$ is an odd superderivation of $J$ and $D^2=0$.

Assume that we have a self-similar Poisson superalgebra $P$ with an embedding~\eqref{selfP}
$\psi:P\hookrightarrow  H\otimes P_1$, where $H$ is a Poisson superalgebra, and $P_1\cong P$.
Denote the Kantor double $J_1=\Kan(P_1)$.
We apply the Kantor double to both algebras in the embedding $\psi$ above
\begin{align*}
J&=\Kan(P)=P\oplus \bar P \hookrightarrow\Kan(H\otimes P_1)= H\otimes P_1\oplus \overline {H\otimes P_1}\\
&\cong H\otimes (P_1\oplus \bar P_1)\cong H\otimes\Kan(P_1)=H\otimes J_1,
\end{align*}
where the isomorphisms in the last line are that of vector spaces.
Let us express the product $\bullet$ of the last space,
which is identified with the Jordan superalgebra $\Kan(H\otimes P_1)$,
in terms of the Jordan product $\circ$ of $J_1=\Kan(P_1)$
and two products $(\ \cdot\ , \{\ , \ \} )$ of the Poisson superalgebra $H$.
Recall that the product $\bullet$ on $\Kan(H\otimes P_1)$ was defined by Kantor's construction.
Take homogeneous $x,y\in H$ and $a,b\in P_1$.
Consider four cases,
where the unspecified products are the associative supercommutative products of the Poisson superalgebra $H\otimes P_1$
\begin{align*}
x a\bullet yb&= xa yb=(-1)^{|a||y|}(xy) (ab)=(-1)^{|a||y|}(x\cdot y)(a\circ b);\\
x a\bullet y\bar b&= xay\bar b=(-1)^{|a||y|}(xy) (a\bar b)=(-1)^{|a||y|}(x\cdot y)(a\circ \bar b);\\
x \bar a\bullet yb&= (-1)^{|yb|}xay \bar b=(-1)^{|y|+|b|+|a||y|} (xy) (a\bar b)=(-1)^{|y|+|a||y|}(x\cdot y)(\bar a\circ b)\\
  &=(-1)^{|\bar a||y|}(x\cdot y)(\bar a\circ b);\\
x \bar a\bullet y\bar b&=(-1)^{|yb|} \{xa,yb\}=(-1)^{|y|+|b|+|a||y|}(xy\{a,b\} +\{x,y\}ab)\\
&=(-1)^{|y|+|b|+|a||y|}\left ((-1)^{|b|} (x\cdot y) (\bar a\circ \bar b)+ (-1)^{|a|+|b|}\{x,y\} (D(\bar a)\circ D(\bar b))\right )\\
&=(-1)^{|\bar a||y|} \left ((x\cdot y)(\bar a\circ \bar b)+(-1)^{|\bar a|+1} \{x,y\} (D(\bar a)\circ D(\bar b))\right).
\end{align*}
Now let $f,g\in J_1=\Kan(P_1)$, belonging to either $P_1$ or $\bar P_1$ and $x,y\in H$. We combine four cases above:
\begin{equation}\label{new_prod}
(x\otimes f)\bullet (y\otimes g)
=(-1)^{|f||y|} \left (x\cdot y\otimes f\circ g+(-1)^{|f|+1} \{x,y\}\otimes  D(f)\circ D(g) \right).
\end{equation}
By these arguments we have the following statement.

\begin{Lemma}
Let $P$ be a self-similar Poisson superalgebra,
i.e. there exist a Poisson superalgebra $H$, which products being $(\ \cdot\ , \{\ , \ \} )$, and a self-similarity embedding~\eqref{selfP}.
Then the Kantor double $J=\Kan(P)$, which is a Jordan superalgebra with a product $\circ$, enjoys the self-similarity embedding
\begin{equation*}%\label{selfP}
J\hookrightarrow  H\otimes J,
\end{equation*}
where the right hand side is a Jordan superalgebra that product $\bullet$ satisfies~\eqref{new_prod},
and the superderivative $D:J\to J$ was defined above.
\end{Lemma}
\begin{Corollary}
Consider the Poisson superalgebra $\PP=\Poisson(A_0,B_0,C_0)$
related to $\QQ$ as above. Let $\JJ=\Kan(\PP)$ be its Kantor double.
Then the Jordan superalgebra $\JJ$ has a  self-similarity embedding
\begin{equation*}%\label{selfP}
\JJ\hookrightarrow  H_3\otimes \JJ.
\end{equation*}
\end{Corollary}

The notion of the {\it wreath product} plays an important role in group theory~\cite{KalKra51,Nekr05}.
Analogous notion of a {\it wreath product} was defined for arbitrary two Lie algebras~\cite{PeRaSh}, see also~\cite{Bartholdi15,FutKochSid}.
Similarly, the notion of the wreath products of associative algebras has many applications, see a recent paper~\cite{AlAlJaZe}.
The observations above allow us to suggest that there exists a notion of
a {\it wreath product} of a Jordan superalgebra with a Poisson superalgebra as follows.
\begin{Conjecture}
Let $J_1$ be a Jordan superalgebra with the product $\circ$, $D:J_1\to J_1$ an odd superderivative such that $D^2=0$.
Let $H$  be a Poisson superalgebra with products $(\ \cdot,\ \{\ ,\ \} )$.
Supply $J=H\otimes J_1$ with the product~\eqref{new_prod}.
Is it true that $J$ is a Jordan superalgebra?
In this case, $J$ should be called the {\em wreath product} of $J_1$ with $H$.
\end{Conjecture}

On the other hand consider a "small" Jordan superalgebra related to the Lie superalgebra $\QQ$ above.
Namely, by~\eqref{defJor} set $\KK=\Jor(\QQ)=\langle 1\rangle \oplus \QQ\oplus \langle \bar 1\rangle \oplus \bar \QQ$.
Then it seems that $\KK$ does not have a self-similarity embedding.
On the other hand, $\KK$ is fractal.

%\bigskip
%
%On behalf of all authors, the corresponding author states that there is no conflict of interest.

%%%%%%%%%%%%%%%%%%%%%%%%%%%%%%%%%%%%%%%%%%%%%%%%%%%%%%%%%%%%%%%%%%%%%%%%%%%%%%%%%%%%%

\end{document}